\documentclass{amsart}
\usepackage{amsthm,amsfonts,amsmath,amssymb}
\usepackage[abs]{overpic}

\newtheorem{theorem}{Theorem}[section]
\newtheorem{lemma}[theorem]{Lemma}

\newtheorem{proposition}[theorem]{Proposition}

\theoremstyle{definition}
\newtheorem{definition}[theorem]{Definition}
\newtheorem{example}[theorem]{Example}

\newtheorem{remark}[theorem]{Remark}

\newcommand{\e}{\varepsilon}

\makeatletter

\@addtoreset{figure}{section}
\makeatother

\makeatletter
  
  \@addtoreset{equation}{section}
\makeatother

\makeatletter
\@namedef{subjclassname@2020}{%
  \textup{2020} Mathematics Subject Classification}
\makeatother

\renewcommand{\labelenumi}{(\roman{enumi})}

\begin{document}
\title[CF-moves for virtual links]{CF-moves for virtual links}

\author[Kodai Wada]{Kodai Wada}
\address{Department of Mathematics, Graduate School of Science, Osaka University, 1-1 Machikaneyama-cho, Toyonaka, Osaka 560-0043, Japan}
\email{ko-wada@cr.math.sci.osaka-u.ac.jp}



\thanks{This work was supported by JSPS KAKENHI Grant Number JP19J00006.}



\begin{abstract}
Oikawa defined an unknotting operation on virtual knots, called a CF-move, and gave a classification of 2-component virtual links up to CF-moves by the virtual linking number and his $n$-invariant. 
In particular, it was proved that a CF-move characterizes the information contained in the virtual linking number for 2-component odd virtual links. 
In this paper, we extend this result by classifying odd virtual links and almost odd virtual links with arbitrary number of components up to CF-moves, using the virtual linking number. 
Moreover, we extend Oikawa's $n$-invariant and introduce two invariants for 3-component even virtual links. 
Using these invariants together with the virtual linking number, we classify 3-component even virtual links up to CF-moves. 
As a result, a classification of 3-component virtual links up to CF-moves is provided. 
\end{abstract}

\maketitle

\section{Introduction}
Virtual links were introduced by Kauffman in~\cite{Kau} as a generalization of classical links in the 3-sphere. 
For an integer $\mu\geq1$, a \emph{$\mu$-component virtual link diagram} is the image of an immersion of $\mu$ circles into the plane, whose singularities are only transverse double points. 
Such double points are divided into \emph{classical crossings} and \emph{virtual crossings} as shown in Figure~\ref{xing}. 
A \emph{$\mu$-component virtual link} is an equivalence class of $\mu$-component virtual link diagrams under \emph{generalized Reidemeister moves}, which consist of classical Reidemeister moves R1--R3 and virtual Reidemeister moves V1--V4 as shown in Figure~\ref{GRmoves}. 
A 1-component virtual link is also called a \emph{virtual knot}. 
Throughout this paper, all virtual links are assumed to be ordered and oriented. 

\begin{figure}[htb]
\centering
  \begin{overpic}[width=4.5cm]{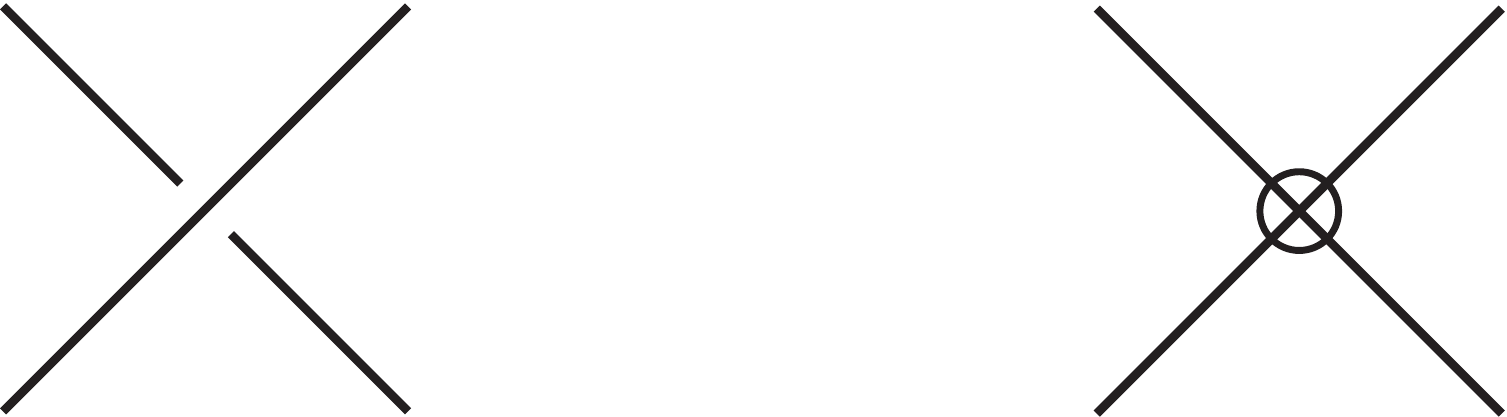}
    \put(-19,-15){classical crossing}
    \put(78,-15){virtual crossing}
  \end{overpic}
\vspace{1em}
\caption{Two types of double points on virtual link diagrams}
\label{xing}
\end{figure}

\begin{figure}[htb]
\centering
  \begin{overpic}[width=12cm]{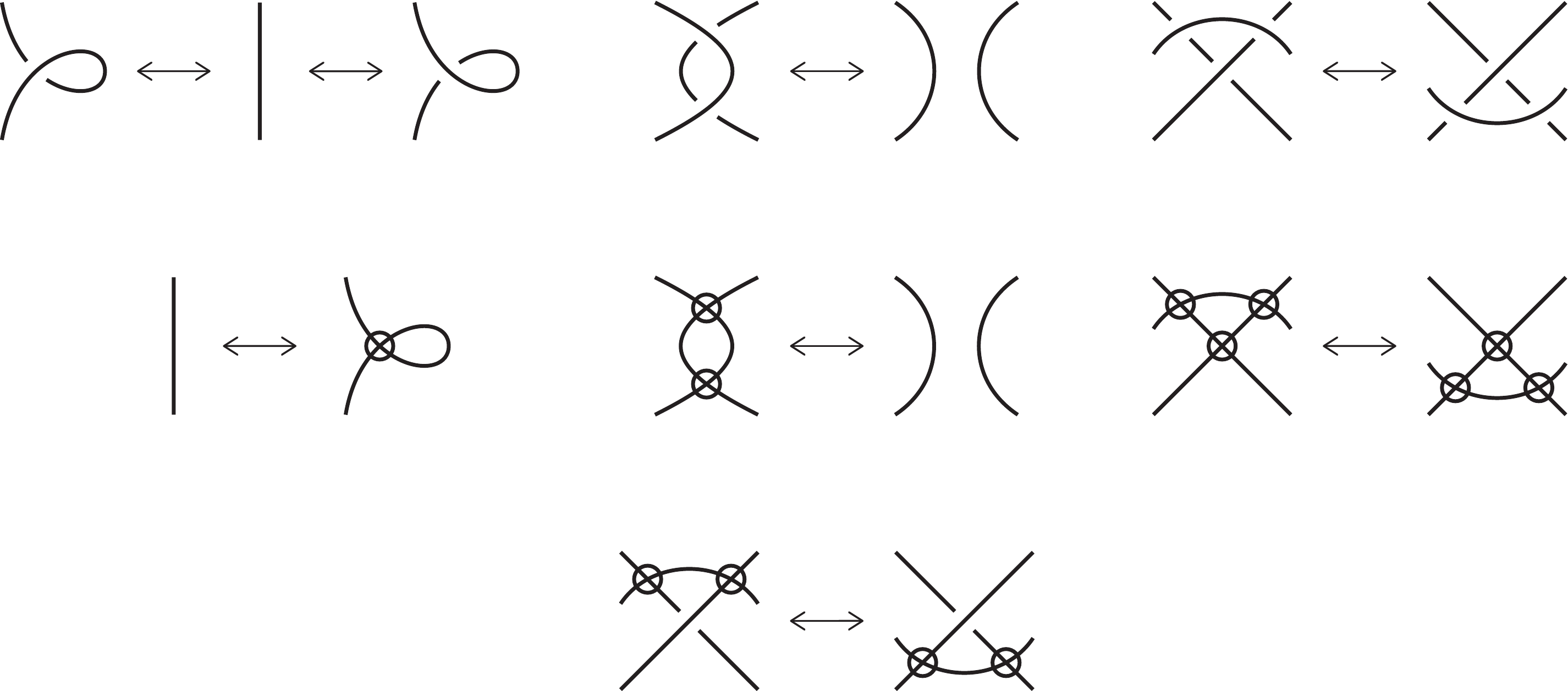}
    \put(32,140.5){R1}
    \put(69.5,140.5){R1}
    \put(174,140.5){R2}
    \put(290,140.5){R3}
    \put(50.5,80.5){V1}
    \put(174,80.5){V2} 
    \put(290,80.5){V3}
    \put(174,20.5){V4}
  \end{overpic}
\caption{Generalized Reidemeister moves}
\label{GRmoves}
\end{figure}
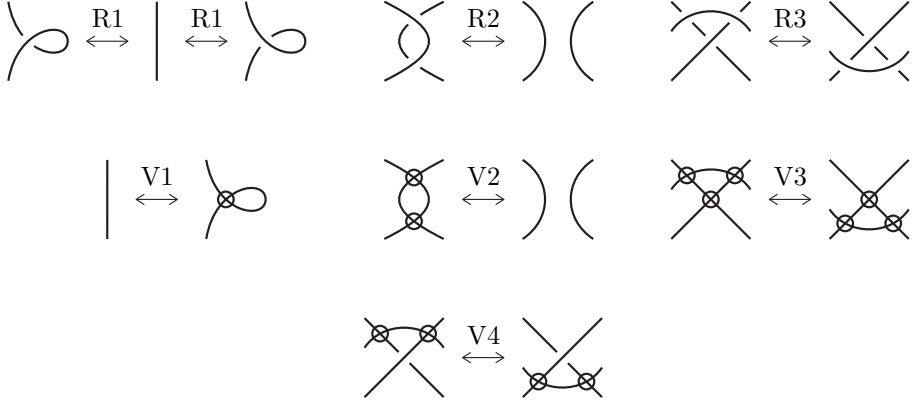

In virtual knot and link theory, we do not allow to use two local moves as shown in Figure~\ref{forbidden}. 
These moves are called the \emph{forbidden moves}. 
Any virtual knot can be deformed into the trivial knot by forbidden moves (cf.~\cite{GPV,Kan,Nel}). 
Generally, the $\mu$-component virtual links $L=K_{1}\cup \cdots\cup K_{\mu}$ up to forbidden moves are classified by the $(i,j)$-linking numbers $\mbox{Lk}(K_{i},K_{j})\in\mathbb{Z}$ $(1\leq i\neq j\leq \mu)$ completely (cf.~\cite{ABMW,Nas,Oka}). 

\begin{figure}[htb]
\centering
  \begin{overpic}[width=9cm]{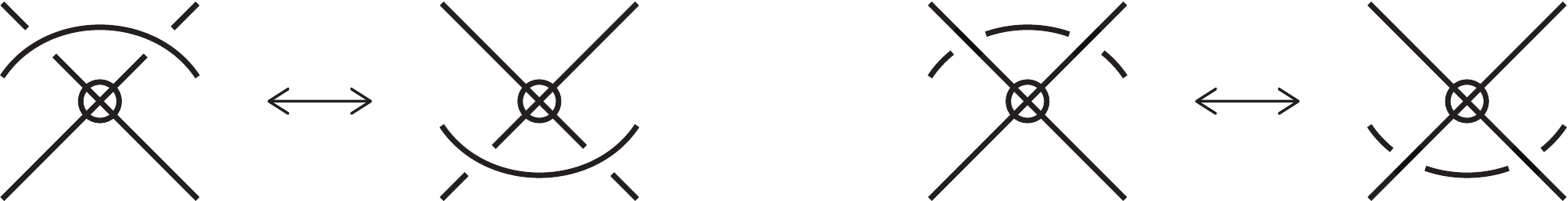}
  \end{overpic}
\caption{Forbidden moves}
\label{forbidden}
\end{figure}

A \emph{CF-move} introduced by Oikawa in~\cite{Oik} is a local move on virtual link diagrams as shown in Figure~\ref{CF}. 
It is considered as a combination of crossing changes and a forbidden move. 
Two virtual links are \emph{CF-equivalent} if their diagrams are related by a finite sequence of generalized Reidemeister moves and CF-moves. 
In~\cite[Theorem~1.2]{Oik}, Oikawa proved that any virtual knot is CF-equivalent to the trivial knot, i.e. a CF-move is an unknotting operation. 
Moreover, he defined an invariant $n(L)\in\mathbb{N}\cup\{0\}$ of a 2-component virtual link $L=K_{1}\cup K_{2}$ with $\mathrm{Lk}(K_{1},K_{2})\equiv\mathrm{Lk}(K_{2},K_{1})\pmod{2}$. 
Then using the virtual linking number $\mathrm{vlk}(K_{1},K_{2})=-\frac{1}{2}(\mathrm{Lk}(K_{1},K_{2})-\mathrm{Lk}(K_{2},K_{1}))$
and $n(L)$, a classification of 2-component virtual links up to CF-equivalence was given in~\cite[Theorem~1.7]{Oik} as follows. 

\begin{figure}[htb]
\centering
  \begin{overpic}[width=4cm]{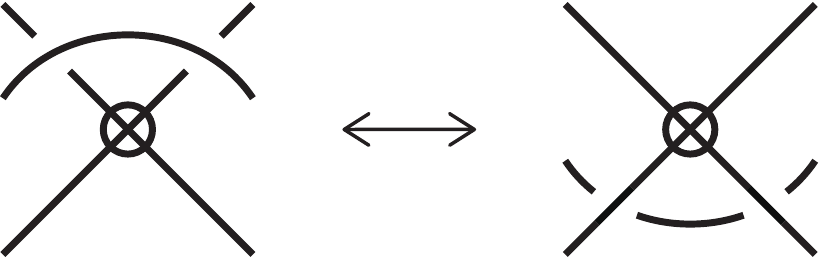}
  \end{overpic}
\caption{A CF-move}
\label{CF}
\end{figure}

\begin{theorem}[{\cite{Oik}}]\label{th-Oikawa}
Let $L=K_{1}\cup K_{2}$ and $L'=K'_{1}\cup K'_{2}$ be $2$-component virtual links. 
\begin{enumerate}
\item Assume that $\mathrm{vlk}(K_{1},K_{2})$ is not an integer. 
Then $L$ and $L'$ are CF-equivalent if and only if $\mathrm{vlk}(K_{1},K_{2})=\mathrm{vlk}(K'_{1},K'_{2})$. 
\item Assume that $\mathrm{vlk}(K_{1},K_{2})$ is an integer. 
Then $L$ and $L'$ are CF-equivalent if and only if $\mathrm{vlk}(K_{1},K_{2})=\mathrm{vlk}(K'_{1},K'_{2})$ 
and $n(L)=n(L')$. 
\end{enumerate}
\end{theorem}

This paper aims to extend this theorem to virtual links with three or more components. 
To this end, we consider the \emph{parity} for every component of a $\mu$-component virtual link $L=K_{1}\cup\cdots\cup K_{\mu}$. 
For $1\leq i\leq\mu$, the $i$th component $K_{i}$ of $L$ is called \emph{odd} if one encounters an odd number of classical crossings in going along the $i$th component of a diagram of $L$; otherwise it is called \emph{even}. 
We say that $L$ is \emph{odd} if every component of $L$ is odd. 
On the other hand, we say that $L$ is \emph{even} if every component of $L$ is even (cf.~\cite{MWY}). 
For example, all virtual knots and classical links are even. 
In the case $\mu=2$, it follows that $L=K_{1}\cup K_{2}$ is odd if and only if $\mathrm{vlk}(K_{1},K_{2})$ is not an integer, 
and that it is even if and only if $\mathrm{vlk}(K_{1},K_{2})$ is an integer. 
There are virtual links with three or more components that are neither odd nor even. 
In particular, we say that $L$ is \emph{almost odd} if some component of $L$ is even and the other components are odd. 
We remark that if $L$ is odd, then $\mu$ is an even integer. 
On the other hand, if $L$ is almost odd, then $\mu$ is an odd integer. 

For odd virtual links and almost odd virtual links, we have the following. 

\begin{theorem}\label{th-odd}
For two $\mu$-component odd or almost odd virtual links $L=K_{1}\cup \cdots\cup K_{\mu}$ and $L'=K'_{1}\cup \cdots\cup K'_{\mu}$, the following are equivalent. 
\begin{enumerate}
\item $L$ and $L'$ are CF-equivalent. 
\item $\mathrm{Lk}(K_{i},K_{j})-\mathrm{Lk}(K_{j},K_{i})=\mathrm{Lk}(K'_{i},K'_{j})-\mathrm{Lk}(K'_{j},K'_{i})$ for any $1\leq i<j\leq\mu$. 
\end{enumerate}
\end{theorem}

By the classification of welded links, which are a quotient of virtual links by the ``over'' forbidden move in the left of Figure~\ref{forbidden}, 
up to crossing changes given in \cite[Theorem 1]{ABMW}, Theorem~\ref{th-odd} means that CF-equivalence coincides with the equivalence relation generated by crossing changes and forbidden moves for odd virtual links and almost odd virtual links. 
Note that the difference $\mathrm{Lk}(K_{i},K_{j})-\mathrm{Lk}(K_{j},K_{i})$ in Theorem~\ref{th-odd} is written as $\mathrm{vlk}_{ij}-\mathrm{vlk}_{ji}$ in~\cite{ABMW}. 
For odd virtual links and almost odd virtual links, a CF-move is indeed a combination of crossing changes and a forbidden move.

For even virtual links, we only consider the case of three components. 
For a 3-component even virtual link $L=K_{1}\cup K_{2}\cup K_{3}$, we define an invariant $n(K_{i},K_{j})\in\mathbb{N}\cup\{0\}$ $(1\leq i< j\leq3)$ as an extension of Oikawa's $n$-invariant given in~\cite{Oik} (Definition~\ref{def-n-inv}). 
Moreover, we introduce two invariants $\tau_{0}(L), \tau_{1}(L)\in\mathbb{Z}$ of $L$ (Definition~\ref{def-tau-inv}). 
Then we have the following. 

\begin{theorem}\label{th-even}
For two $3$-component even virtual links $L=K_{1}\cup K_{2}\cup K_{3}$ and $L'=K'_{1}\cup K'_{2}\cup K'_{3}$, the following are equivalent. 
\begin{enumerate}
\item $L$ and $L'$ are CF-equivalent. 
\item $\mathrm{Lk}(K_{i},K_{j})-\mathrm{Lk}(K_{j},K_{i})=\mathrm{Lk}(K'_{i},K'_{j})-\mathrm{Lk}(K'_{j},K'_{i})$ and 
$n(K_{i},K_{j})=n(K'_{i},K'_{j})$ for any $1\leq i<j\leq3$, 
and $\tau_{0}(L)-\tau_{1}(L)=\tau_{0}(L')-\tau_{1}(L')$. 
\end{enumerate}
\end{theorem}

Since a 3-component virtual link is either almost odd or even, 
a complete classification of 3-component virtual links up to CF-equivalence is provided by Theorems~\ref{th-odd} and~\ref{th-even}.

\section{Gauss diagrams up to CF-moves}\label{sec-Gauss}
In this section, we review the definition of Gauss diagrams and give a normal form of a $\mu$-component virtual link up to CF-equivalence in terms of Gauss diagrams (Proposition~\ref{prop-NF}). 

A \emph{Gauss diagram} is a disjoint union of ordered and oriented circles together with signed and oriented chords whose endpoints lie disjointly on the circles. 
Throughout this paper, all circles of a Gauss diagram are assumed to be  oriented counterclockwise and arranged in increasing order from left to right. 
A chord in a Gauss diagram is called a \emph{self-chord} if both endpoints lie on the same circle of the Gauss diagram; 
otherwise it is called a \emph{nonself-chord}. 
In particular, a self-chord is \emph{free} if its endpoints are adjacent on the circle. 
A nonself-chord is \emph{of type $(i,j)$} if it is oriented from the $i$th circle to the $j$th one $(i\neq j)$. 

Given a $\mu$-component virtual link diagram with $n$ classical crossings, its Gauss diagram is defined to be the union of $\mu$~circles and $n$~chords connecting the preimage of each classical crossing. 
Each chord is equipped with the sign of the corresponding classical crossing, and oriented from the over-crossing to the under-crossing. 
By definition, the virtual Reidemeister moves V1--V4 on virtual link diagrams do not affect the corresponding Gauss diagrams. 
On the other hand, the classical Reidemeister moves R1--R3 change the Gauss diagrams as shown in Figure~\ref{Rmoves-Gauss}. 
Here, there are several kinds of R3 depending on the signs and orientations of chords, and one of them is shown in the figure. 
The others are generated by R1, R2, and this R3 (cf.~\cite{P}). 
Therefore, a virtual link can be considered as an equivalence class of Gauss diagrams under Reidemeister moves R1--R3 (cf.~\cite{GPV,Kau}). 

\begin{figure}[htb]
  \begin{center}
    \begin{overpic}[width=12cm]{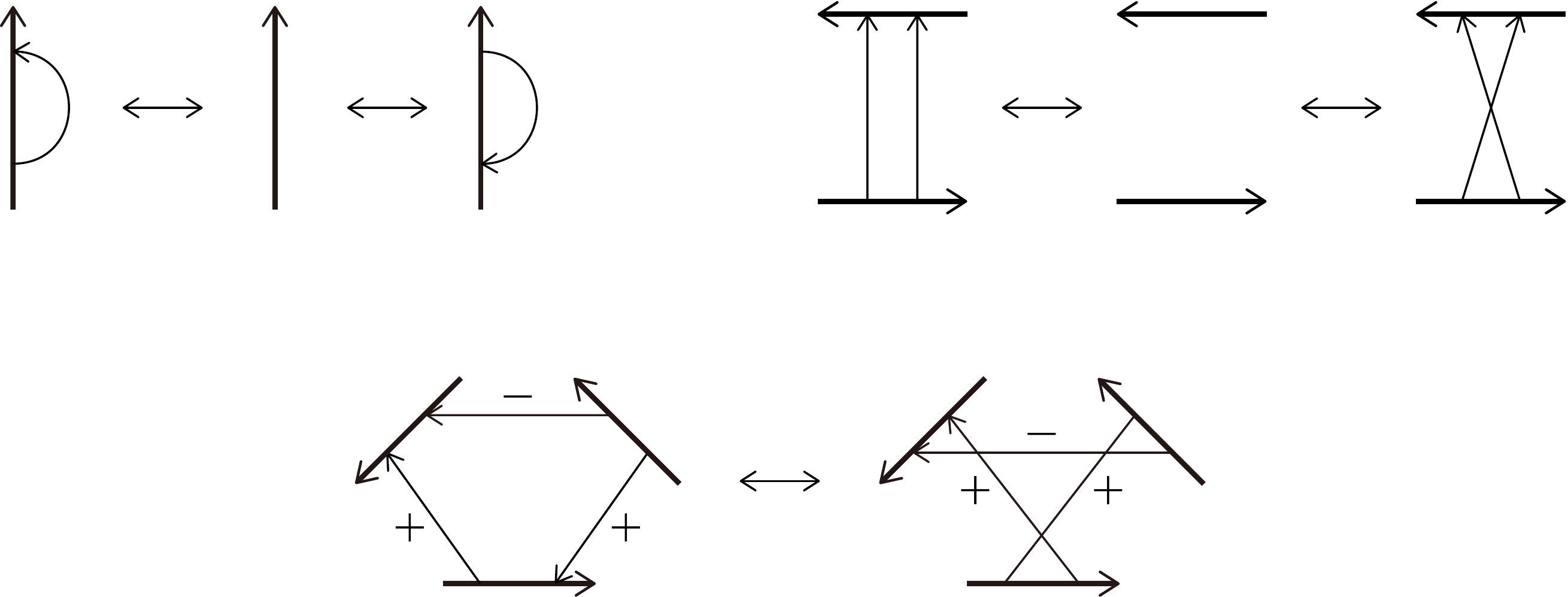}
      \put(30,112){R1}
      \put(78.5,112){R1}
      \put(221,112){R2}
      \put(286,112){R2}
      \put(164,30.5){R3}
      \put(15,95){$\e$}
      \put(117,95){$\e$}
      \put(180,95){$\e$}
      \put(203,95){$-\e$}
      \put(313,95){$\e$}
      \put(331,95){$-\e$}
    \end{overpic}
  \end{center}
  \caption{Reidemeister moves on Gauss diagrams with $\e\in\{\pm1\}$}
  \label{Rmoves-Gauss}
\end{figure}

A \emph{CF-move} on Gauss diagrams is a local move as shown in Figure~\ref{CF-Gauss}. 
Two Gauss diagrams $G$ and $G'$ are \emph{CF-equivalent} if they are related by a finite sequence of Reidemeister moves R1--R3 and CF-moves. 
It is denoted by $G\sim G'$. 
We emphasize that two virtual links are CF-equivalent if and only if their Gauss diagrams are CF-equivalent. 

\begin{figure}[htb]
\centering
  \begin{overpic}[width=6cm]{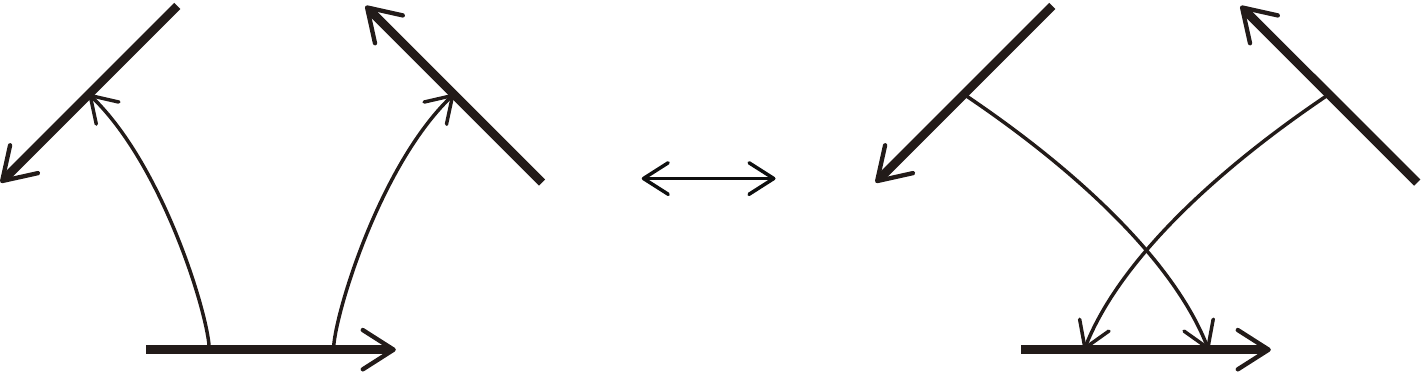}
    \put(11,17){$\e$}
    \put(50,17){$\e'$}
    \put(114,17){$-\e$}
    \put(148,17){$-\e'$}
  \end{overpic}
\caption{A CF-move on Gauss diagrams with $\e, \e'\in\{\pm1\}$}
\label{CF-Gauss}
\end{figure}

\begin{lemma}\label{lem-HT}
If two Gauss diagrams are related by a local move as shown in Figure~\ref{HT}, then they are CF-equivalent. 
\end{lemma}

\begin{figure}[htb]
\centering
  \begin{overpic}[width=6cm]{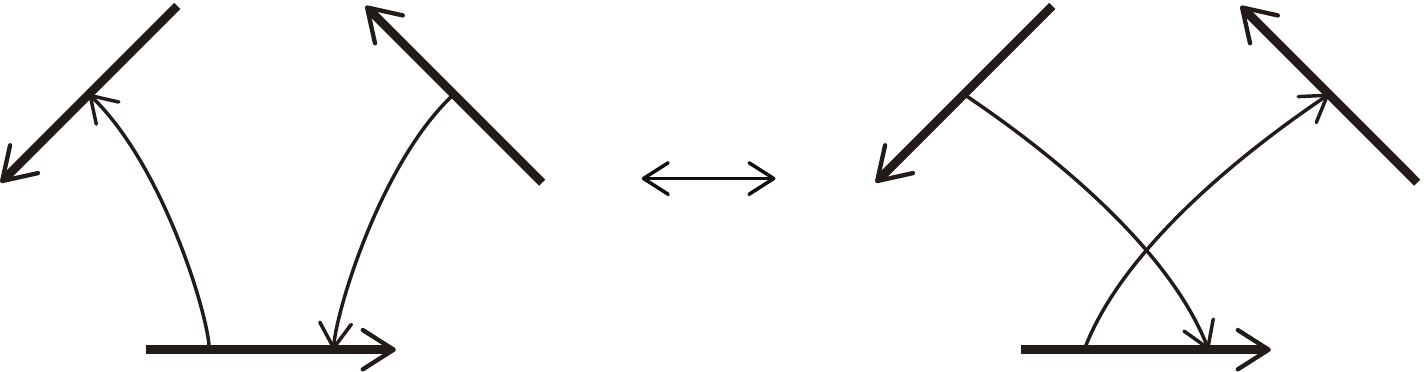}
    \put(11,17){$\e$}
    \put(50,17){$\e'$}
    \put(114,17){$-\e$}
    \put(148,17){$-\e'$}
  \end{overpic}
\caption{Exchanging two consecutive initial and terminal endpoints of a pair of chords}
\label{HT}
\end{figure}
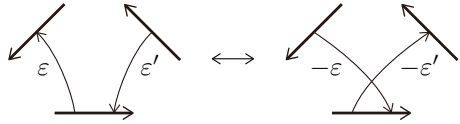

\begin{proof}
The proof in the case $\e\e'=1$ follows from Figure~\ref{pf-lem-HT-same}, and that in the case $\e\e'=-1$ follows from Figure~\ref{pf-lem-HT-opposite}.
\end{proof}

\begin{figure}[htb]
\centering
  \begin{overpic}[width=10cm]{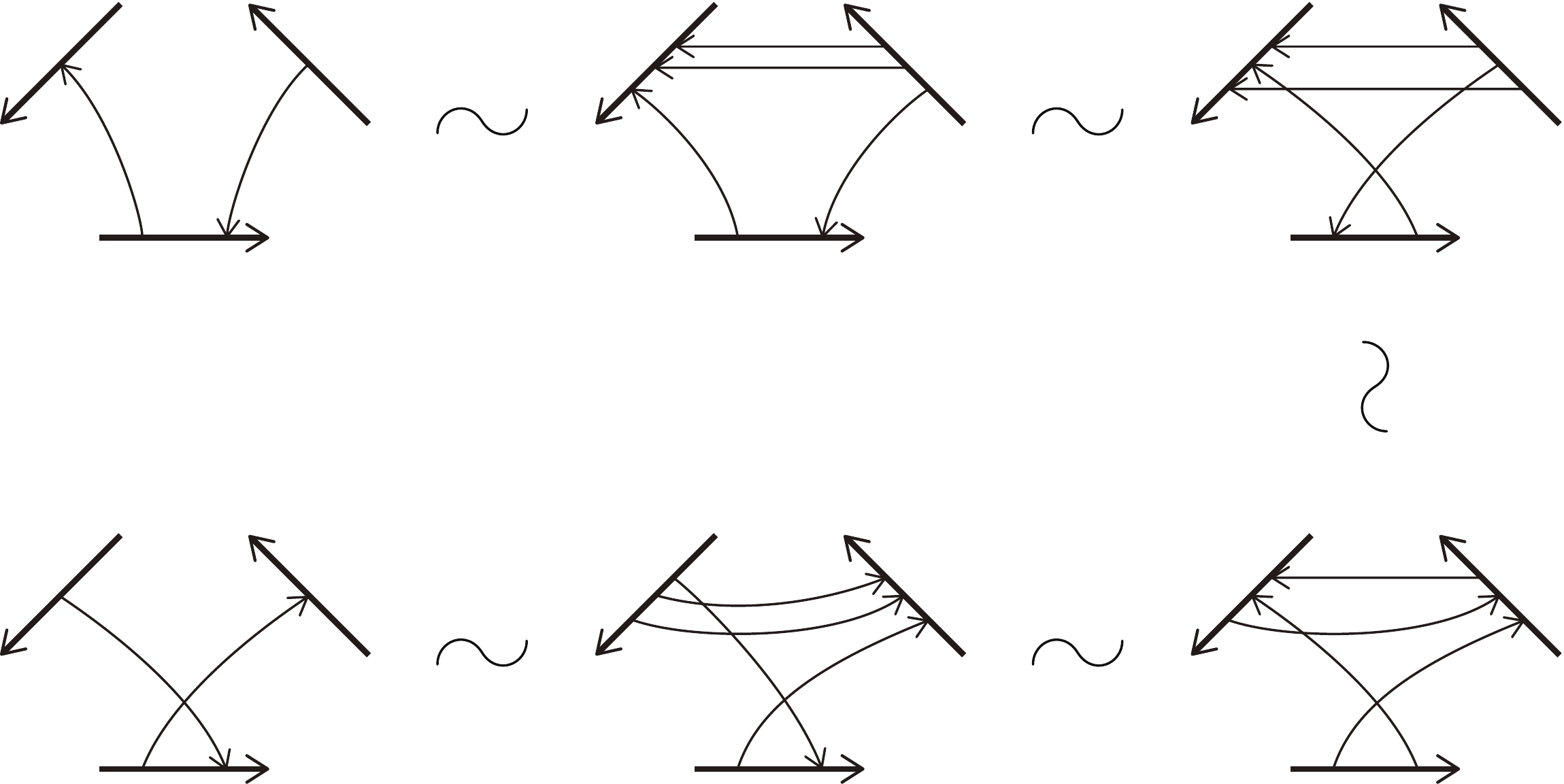}
    \put(81.5,127){R2}
    \put(189.5,127){R3}
    \put(256,68){CF}
    \put(81.5,30){R2}
    \put(189.5,30){CF}
    \put(12,114){$\e$}
    \put(50,114){$\e$}
    \put(116,114){$\e$}
    \put(162,114){$\e$}
    \put(140,137){$\e$}
    \put(135,122.5){$-\e$}
    \put(248.5,137){$\e$}
    \put(243.5,128.2){$-\e$}
    \put(236,114){$\e$}
    \put(259,114){$\e$}
    \put(7,21){$-\e$}
    \put(46,21){$-\e$}
    \put(122,18){$-\e$}
    \put(155,18){$-\e$}
    \put(138,36){$-\e$}
    \put(143,21.5){$\e$}
    \put(235,18){$\e$}
    \put(263,18){$-\e$}
    \put(248.5,40){$\e$}
    \put(248.5,29.5){$\e$}
  \end{overpic}
\caption{Proof of Lemma~\ref{lem-HT} in the case $\e\e'=1$}
\label{pf-lem-HT-same}
\end{figure}

\begin{figure}[htb]
\centering
  \begin{overpic}[width=10cm]{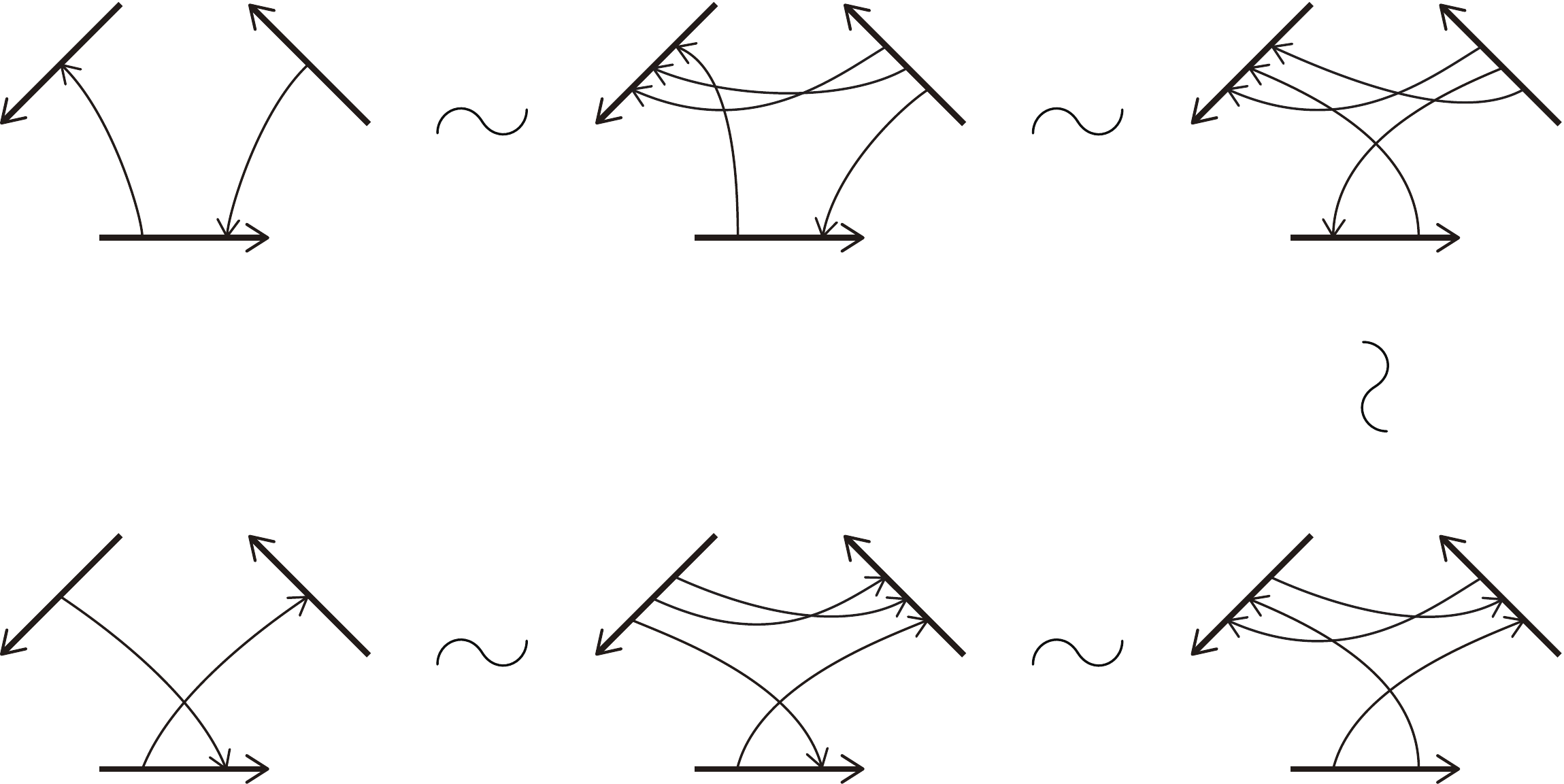}
    \put(81.5,127){R2}
    \put(189.5,127){R3}
    \put(256,68){CF}
    \put(81.5,30){R2}
    \put(189.5,30){CF}
    \put(12,114){$\e$}
    \put(50,114){$-\e$}
    \put(126,107){$\e$}
    \put(157,107){$-\e$}
    \put(149,132){$\e$}
    \put(148,120.5){$-\e$}
    \put(258,132){$\e$}
    \put(262,117.5){$-\e$}
    \put(229,107){$-\e$}
    \put(260,107){$\e$}
    \put(7,21){$-\e$}
    \put(48,21){$\e$}
    \put(130,9){$\e$}
    \put(149,9){$-\e$}
    \put(132.5,35.5){$\e$}
    \put(141.5,35.5){$-\e$}
    \put(239,9){$\e$}
    \put(260,9){$\e$}
    \put(258.5,35.5){$\e$}
    \put(240,35.5){$\e$}
  \end{overpic}
\caption{Proof of Lemma~\ref{lem-HT} in the case $\e\e'=-1$}
\label{pf-lem-HT-opposite}
\end{figure}

\begin{remark}\label{rem-Oik}
By Lemma~\ref{lem-HT}, any two consecutive endpoints of a pair of chords $\gamma$ and $\gamma'$ can be exchanged up to CF-equivalence, 
although the signs and orientations of $\gamma$ and $\gamma'$ are altered. 
Therefore, up to CF-equivalence, we can deform every self-chord in a Gauss diagram into a free chord, and remove it by an R1-move. 
In particular, any Gauss diagram of a virtual knot is CF-equivalent to the one without chords, which represents the trivial knot. 
This is an alternative proof of~\cite[Theorem~1.2]{Oik} in terms of Gauss diagrams. 
\end{remark}

For $1\leq i\neq j\leq\mu$ and an integer $a_{ij}\in\mathbb{Z}$, let $G_{ij}(a_{ij})$ be the Gauss diagram with $\mu$ circles $C_{1},\ldots,C_{\mu}$, which consists of only $|a_{ij}|$ horizontal nonself-chords of type~$(i,j)$ with sign $\e_{ij}$, where $a_{ij}=\e_{ij}|a_{ij}|$. 
Let $G$ and $G'$ be Gauss diagrams with $\mu$ circles $C_{1},\ldots,C_{\mu}$ and $C'_{1},\ldots,C'_{\mu}$, respectively. 
We put $C_{1},\ldots,C_{\mu}$ on top and $C'_{1},\ldots,C'_{\mu}$ on bottom. 
As shown in Figure~\ref{surgery}, for every $i\in\{1,\ldots,\mu\}$, we connect $C_{i}$ and $C'_{i}$ by an unknotted arc, and then form a new circle $C_{i}\# C'_{i}$ by surgery along the unknotted arc. 
The Gauss diagram consisting of $\mu$ circles $C_{1}\# C'_{1}, \ldots,C_{\mu}\# C'_{\mu}$ is called a \emph{connected sum} of $G$ and $G'$ and denoted by $G\# G'$. 
An example is shown in Figure~\ref{ex-sum}.

\begin{figure}[htb]
\centering
  \begin{overpic}[width=5cm]{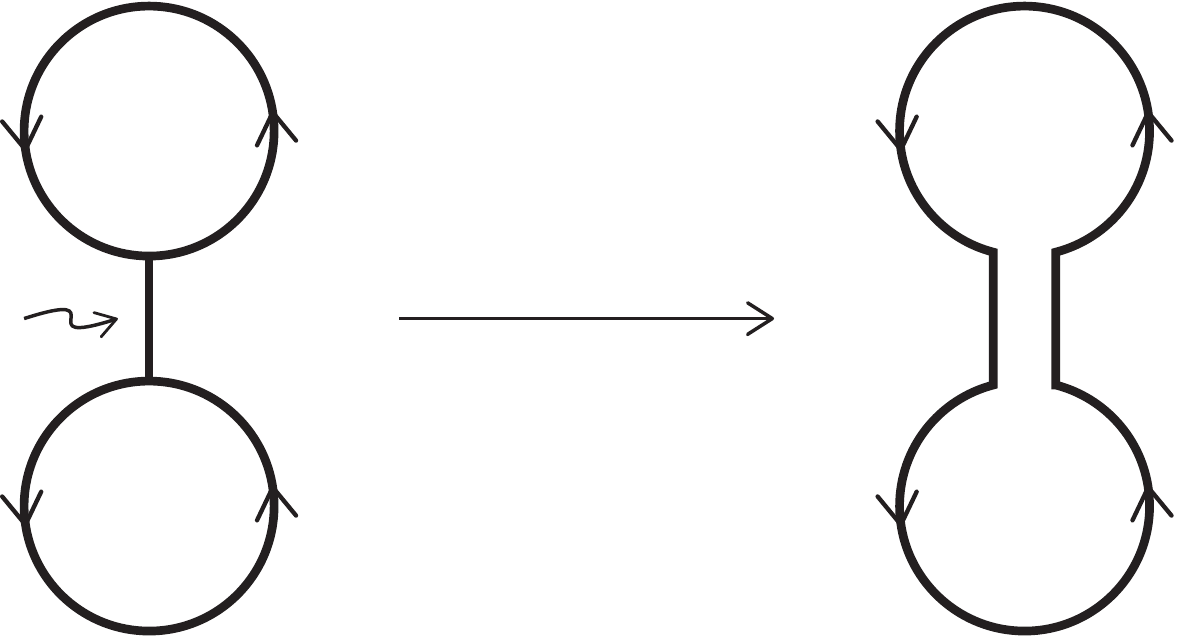}
    \put(13.5,58){$C_{i}$}
    \put(13,12.5){$C'_{i}$}
    \put(-45,36){unknotted}
    \put(-28,26){arc}
    \put(54,44){surgery}
    \put(148,35){$C_{i}\# C'_{i}$}
  \end{overpic}
\caption{Forming a new circle $C_{i}\# C'_{i}$ from two circles $C_{i}$ and $C'_{i}$ by surgery along an unknotted arc}
\label{surgery}
\end{figure}

\begin{figure}[htb]
\centering
\begin{overpic}[width=5.5cm]{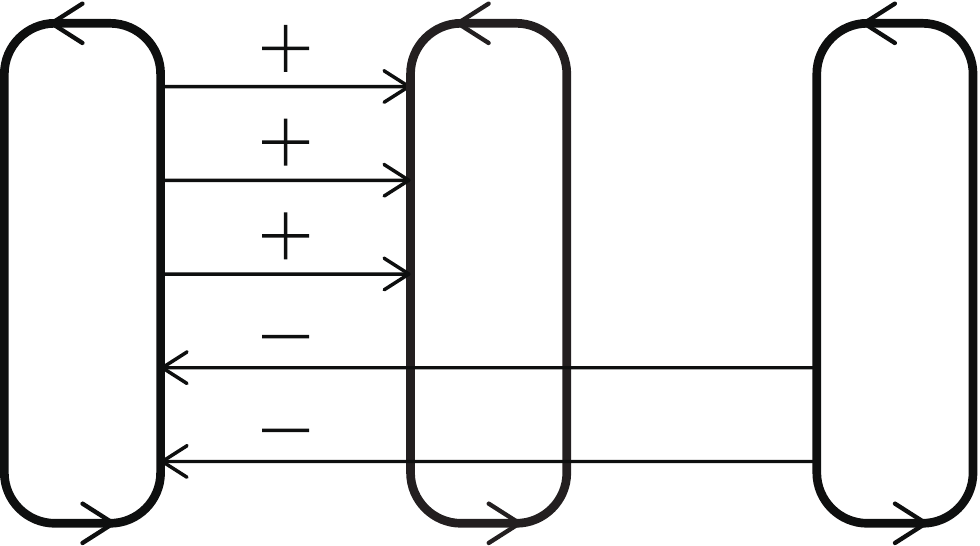}
\end{overpic}
\caption{A connected sum $G_{12}(3)\# G_{31}(-2)$ of two Gauss diagrams $G_{12}(3)$ and $G_{31}(-2)$ in the case $\mu=3$}
\label{ex-sum}
\end{figure}

\begin{proposition}\label{prop-NF}
Any Gauss diagram $G$ of a $\mu$-component virtual link is CF-equivalent to $\#_{1\leq i<j\leq\mu}\left(G_{ij}(a_{ij})\# G_{ji}(a_{ji})\right)$ for some $a_{ij}, a_{ji}\in\mathbb{Z}$, where the connected sum runs over all $i\neq j\in\{1,\ldots,\mu\}$ in lexicographic order of $i$ and $j$. 
\end{proposition}

\begin{proof}
By Remark~\ref{rem-Oik}, we may assume that every circle $C_{i}$ of $G$ has no self-chords up to CF-equivalence. 
For the nonself-chords in $G$, we move the endpoints by Lemma~\ref{lem-HT} to obtain parallel nonself-chords connecting $C_{i}$ and $C_{j}$ that are arranged in lexicographic order of $i\neq j\in\{1,\ldots,\mu\}$ from top to bottom.
For each $1\leq i<j\leq\mu$, we make the parallel nonself-chords connecting $C_{i}$ and $C_{j}$ so that the nonself-chords of type~$(i,j)$ are arranged above those of type~$(j,i)$. 
If two consecutive nonself-chords of the same type have opposite signs, then we cancel them by an R2-move. 
Finally $G$ is CF-equivalent to $\#_{1\leq i<j\leq\mu}\left(G_{ij}(a_{ij})\# G_{ji}(a_{ji})\right)$ for some $a_{ij}, a_{ji}\in\mathbb{Z}$. 
\end{proof}

\section{Proof of Theorem~\ref{th-odd}}\label{sec-odd}
Let $L=K_{1}\cup\cdots\cup K_{\mu}$ be a $\mu$-component virtual link and $G$ its Gauss diagram with $\mu$ circles $C_{1},\ldots,C_{\mu}$. 
For $1\leq i\neq j\leq\mu$, the \emph{$(i,j)$-linking number} of $L$ is defined to be the sum of signs of all nonself-chords of type~$(i,j)$ in $G$. 
It is denoted by $\mathrm{Lk}(K_{i},K_{j})$. 
This integer $\mathrm{Lk}(K_{i},K_{j})$ is an invariant of the virtual link $L$ (cf.~\cite[Section 1.7]{GPV}). 
It can be seen that the difference $\mathrm{Lk}(K_{i},K_{j})-\mathrm{Lk}(K_{j},K_{i})$ is equal to $-2\mathrm{vlk}(K_{i},K_{j})$, 
where $\mathrm{vlk}(K_{i},K_{j})$ denotes virtual linking number of $K_{i}$ and $K_{j}$ given in \cite{Oik,Oka}. 
Therefore, the following result is obtained from Theorem~\ref{th-Oikawa}. 

\begin{lemma}\label{lem-Lk}
For any $1\leq i< j\leq\mu$, the difference $\mathrm{Lk}(K_{i},K_{j})-\mathrm{Lk}(K_{j},K_{i})$ is invariant under CF-moves. 
\end{lemma}

For $1\leq i\leq\mu$, the $i$th circle $C_{i}$ of $G$ is called \emph{odd} if the number of all endpoints of self-/nonself-chords on $C_{i}$ is an odd integer; 
otherwise it is called \emph{even}. 
Since Reidemeister moves and CF-moves do not change the parity of the number of endpoints on $C_{i}$, 
the parity of $C_{i}$ is preserved under CF-equivalence.
We emphasize that the $i$th component $K_{i}$ of $L$ is odd if and only if $C_{i}$ is odd, and that it is even if and only if $C_{i}$ is even. 

\begin{lemma}\label{lem-cc}
Let $G$ be a Gauss diagram with an odd circle $C$. 
Let $G'$ be a Gauss diagram obtained from $G$ by altering the sign and orientation of a nonself-chord $\gamma$ that is attached to $C$. 
Then $G$ and $G'$ are CF-equivalent. 
\end{lemma}

\begin{proof}
A CF-equivalence from $G$ to $G'$ is given as follows. 
First, we use Lemma~\ref{lem-HT} to slide the endpoint of $\gamma$ on $C$ along the circle $C$ until it returns to the original position. 
Since $C$ is odd, $\gamma$ encounters an even number of endpoints during this process. 
Therefore, $\gamma$ still has the same sign and orientation. 
On the other hand, the signs and orientations of the other nonself-chords that are attached to $C$ are altered. 
Note that those of all self-chords in $C$ are preserved. 
Next, we add a free chord in $C$ by an R1-move, and slide one endpoint of the free chord along $C$ until it is next to the other endpoint by Lemma~\ref{lem-HT}. 
Then the signs and orientations of all nonself-chords that are attached to $C$, including $\gamma$, are altered, and those of all self-chords in $C$ are preserved. 
Finally, we delete the free chord by an R1-move. 
The obtained Gauss diagram is $G'$. 
\end{proof}

\begin{proposition}\label{prop-NF-odd}
Any Gauss diagram $G$ of a $\mu$-component odd or almost odd virtual link $L=K_{1}\cup\cdots \cup K_{\mu}$ is CF-equivalent to $\#_{1\leq i<j\leq\mu}G_{ij}(b_{ij})$ for some $b_{ij}\in\mathbb{Z}$. 
Moreover, we have $\mathrm{Lk}(K_{i},K_{j})-\mathrm{Lk}(K_{j},K_{i})=b_{ij}$ for any $1\leq i<j\leq\mu$. 
\end{proposition}

\begin{proof}
Since the proofs in the cases of odd virtual links and almost odd virtual links are similar, we only consider the case where $L$ is almost odd. 
Without loss of generality, we may assume that the $\mu$th circle $C_{\mu}$ of $G$ is even and the other circles $C_{1},\ldots,C_{\mu-1}$ are odd. 

By Proposition~\ref{prop-NF}, $G$ is CF-equivalent to $\#_{1\leq i<j\leq\mu}\left(G_{ij}(a_{ij})\# G_{ji}(a_{ji})\right)$ for some $a_{ij}, a_{ji}\in\mathbb{Z}$. 
Since $C_{1},\ldots,C_{\mu-1}$ are odd, we apply Lemma~\ref{lem-cc} to nonself-chords that are attached to $C_{i}$ for $1\leq i\leq\mu-1$ in order to obtain $\#_{1\leq i<j\leq\mu}\left(G_{ij}(a_{ij})\# G_{ij}(-a_{ji})\right)$. 
Moreover, we obtain $\#_{1\leq i<j\leq\mu}G_{ij}(a_{ij}-a_{ji})$ from it by R2-moves. 
Therefore, $G$ is CF-equivalent to $\#_{1\leq i<j\leq\mu}G_{ij}(b_{ij})$ for $b_{ij}=a_{ij}-a_{ji}\in\mathbb{Z}$. 

Let $L'=K'_{1}\cup\cdots\cup K'_{\mu}$ be the $\mu$-component almost odd virtual link represented by $\#_{1\leq i<j\leq\mu}G_{ij}(b_{ij})$. 
Since the sum of signs of all nonself-chords of type~$(i,j)$ in $\#_{1\leq i<j\leq\mu}G_{ij}(b_{ij})$ is equal to $b_{ij}$ and that of type~$(j,i)$ is equal to zero,
we have $\mathrm{Lk}(K'_{i},K'_{j})-\mathrm{Lk}(K'_{j},K'_{i})=b_{ij}$ $(1\leq i<j\leq\mu)$. 
Since $L$ and $L'$ are CF-equivalent, Lemma~\ref{lem-Lk} implies that $\mathrm{Lk}(K_{i},K_{j})-\mathrm{Lk}(K_{j},K_{i})=b_{ij}$. 
\end{proof}

\begin{proof}[Proof of Theorem~\ref{th-odd}]
$\mathrm{(i)}\Rightarrow\mathrm{(ii)}$: 
This follows from Lemma~\ref{lem-Lk} directly. 

\noindent 
$\mathrm{(ii)}\Rightarrow\mathrm{(i)}$: 
By Proposition~\ref{prop-NF-odd}, any Gauss diagrams of $L$ and $L'$ are CF-equivalent to $\#_{1\leq i<j\leq\mu}G_{ij}(b_{ij})$ and $\#_{1\leq i<j\leq\mu}G_{ij}(b'_{ij})$ for some $b_{ij}, b'_{ij}\in\mathbb{Z}$, respectively. 
Moreover, we have $\mathrm{Lk}(K_{i},K_{j})-\mathrm{Lk}(K_{j},K_{i})=b_{ij}$ and $\mathrm{Lk}(K'_{i},K'_{j})-\mathrm{Lk}(K'_{j},K'_{i})=b'_{ij}$ $(1\leq i<j\leq\mu)$. 
Since $\mathrm{Lk}(K_{i},K_{j})-\mathrm{Lk}(K_{j},K_{i})=\mathrm{Lk}(K'_{i},K'_{j})-\mathrm{Lk}(K'_{j},K'_{i})$, we have $b_{ij}=b'_{ij}$, and therefore $\#_{1\leq i<j\leq\mu}G_{ij}(b_{ij})=\#_{1\leq i<j\leq\mu}G_{ij}(b'_{ij})$. 
\end{proof}

\section{Proof of Theorem~\ref{th-even}} 
Throughout this section, we consider a 3-component \emph{even} virtual link $L=K_{1}\cup K_{2}\cup K_{3}$ and its Gauss diagram $G$ with three circles $C_{1}$, $C_{2}$, and $C_{3}$. 

First we extend Oikawa's $n$-invariant given in~\cite[Definition 1.6]{Oik} to the 3-component even virtual link $L$ from the Gauss diagram point of view. 
We use the equivalence relation among the nonself-chords in $G$ connecting $C_{i}$ and $C_{j}$ $(1\leq i< j\leq3)$ defined in~\cite[Section~6]{MSW}. 
Let $\gamma$ and $\gamma'$ be nonself-chords connecting $C_{i}$ and $C_{j}$. 
The endpoints of $\gamma$ and $\gamma'$ on $C_{i}$ divide the circle $C_{i}$ into two arcs. 
Let $\alpha_{i}$ be one of the two arcs. 
Similarly, the endpoints of $\gamma$ and $\gamma'$ on $C_{j}$ divide $C_{j}$ into two arcs, and let $\alpha_{j}$ be one of the two arcs. 
See Figure~\ref{arcs-pair}. 
We say that $\gamma$ and $\gamma'$ are \emph{equivalent} if the number of all endpoints of self-/nonself-chords on $\alpha_{i}\cup\alpha_{j}$ is an even integer. 
In particular, $\gamma$ is equivalent to itself. 
Since $C_{i}$ and $C_{j}$ are even, the equivalence relation between $\gamma$ and $\gamma'$ does not depend on the choice of arcs $\alpha_{i}$ and $\alpha_{j}$. 

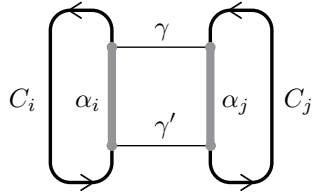
\begin{figure}[htb]
\centering
  \begin{overpic}[width=3cm]{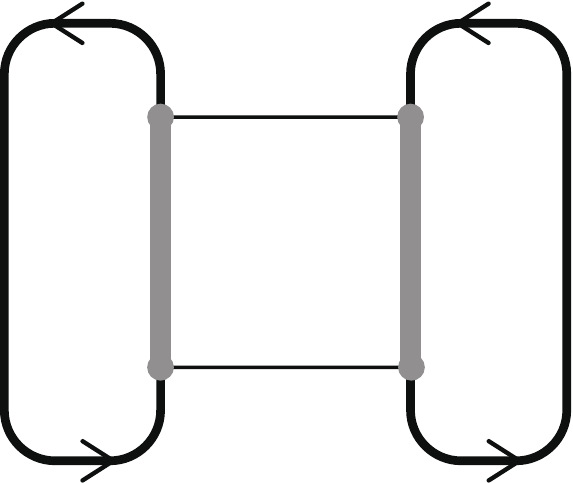}
    \put(10,32){$\alpha_{i}$}
    \put(65.5,32){$\alpha_{j}$}
    \put(-15,32){$C_{i}$}
    \put(89,32){$C_{j}$}
    \put(40,59){$\gamma$}
    \put(40,22){$\gamma'$}
  \end{overpic}
\caption{A pair of arcs $\alpha_{i}$ and $\alpha_{j}$ for nonself-chords $\gamma$ and $\gamma'$ connecting $C_{i}$ and $C_{j}$}
\label{arcs-pair}
\end{figure}

For $1\leq i< j\leq3$, we fix a nonself-chord $\gamma_{0}$ connecting $C_{i}$ and $C_{j}$. 
Let $\sigma(C_{i},C_{j};\gamma_{0})$ be the sum of signs of all nonself-chords connecting $C_{i}$ and $C_{j}$, which are equivalent to $\gamma_{0}$, including $\gamma_0$ itself. 
On the other hand, let $\overline{\sigma}(C_{i},C_{j};\gamma_{0})$ be the sum of signs of all nonself-chords connecting $C_{i}$ and $C_{j}$, which are not equivalent to $\gamma_{0}$. 
We remark that $\sigma(C_{i},C_{j};\gamma_{0})+\overline{\sigma}(C_{i},C_{j};\gamma_{0})=\mathrm{Lk}(K_{i},K_{j})+\mathrm{Lk}(K_{j},K_{i})$. 
Denote by $n(C_{i},C_{j};\gamma_{0})$ the absolute value of the difference $\sigma(C_{i},C_{j};\gamma_{0})-\overline{\sigma}(C_{i},C_{j};\gamma_{0})$. 
Note that $n(C_{i},C_{j};\gamma_{0})=n(C_{j},C_{i};\gamma_{0})$ by definition. 

\begin{example}\label{ex-n-inv}
Let $G$ be the Gauss diagram with three circles $C_{1},C_{2}$, and $C_{3}$ in Figure~\ref{ex-Gauss-even}, which represents a 3-component even virtual link. 
Let $\gamma_{0},\ldots,\gamma_{8}$ be the nonself-chords in $G$ as shown in the figure. 
Choose $\gamma_{0}$ as a fixed nonself-chord connecting $C_{1}$ and $C_{2}$.  
By definition, $\gamma_{0}$ and $\gamma_{1}$ are equivalent to $\gamma_{0}$. 
On the other hand, $\gamma_{2}, \gamma_{3}$, and $\gamma_{4}$ are not equivalent to $\gamma_{0}$. 
Therefore, we have $n(C_{1},C_{2};\gamma_{0})=3$. 
Similarly, choosing $\gamma_{5}$ as a fixed nonself-chord connecting $C_{1}$ and $C_{3}$, we have $n(C_{1},C_{3};\gamma_{5})=3$. 
Moreover, since $\gamma_{8}$ is the only one nonself-chord connecting $C_{2}$ and $C_{3}$, it follows that $n(C_{2},C_{3};\gamma_{8})=1$. 
\end{example}

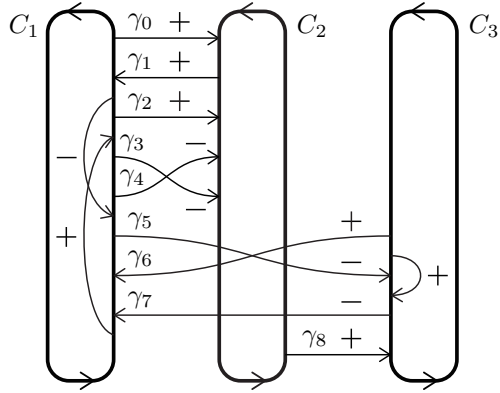
\begin{figure}[htb]
\centering
\begin{overpic}[width=5.5cm]{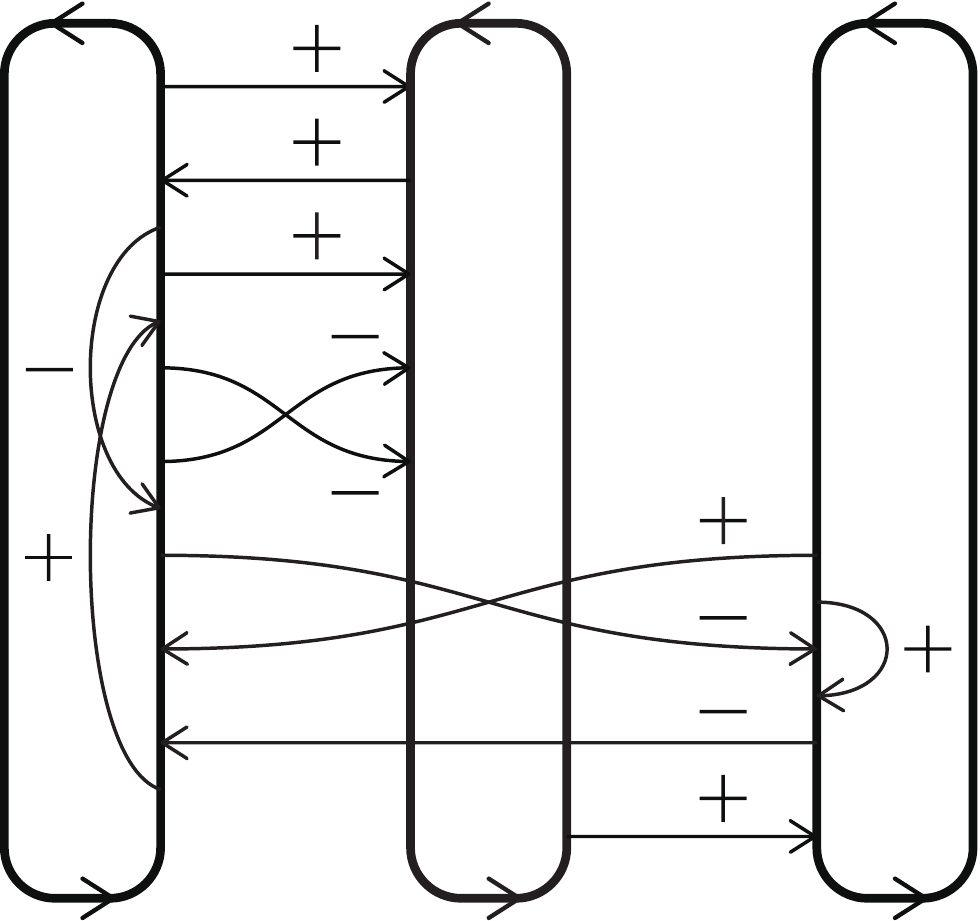}
  \put(-14,135){$C_{1}$}
  \put(95,135){$C_{2}$}
  \put(160,135){$C_{3}$}
  \put(31,138){$\gamma_{0}$}
  \put(31,123){$\gamma_{1}$}
  \put(31,108){$\gamma_{2}$}
  \put(28,92.5){$\gamma_{3}$}
  \put(28,79){$\gamma_{4}$}
  \put(31,63){$\gamma_{5}$}
  \put(31,48.5){$\gamma_{6}$}
  \put(31,33){$\gamma_{7}$}
  \put(97,18){$\gamma_{8}$}
\end{overpic}
\caption{A Gauss diagram representing a 3-component even virtual link}
\label{ex-Gauss-even}
\end{figure}

\begin{lemma}\label{lem-n-inv}
For any $1\leq i< j\leq3$, the nonnegative integer $n(C_{i},C_{j};\gamma_{0})$ is an invariant of $L$. 
Moreover, it is invariant under CF-moves. 
\end{lemma}

Since the proof is obtained from the proofs of Proposition~3.1 and Lemma~3.2 in~\cite{Oik} by interpreting in terms of Gauss diagrams, we omit it here. 
By Lemma~\ref{lem-n-inv}, the following definition is well-defined. 

\begin{definition}\label{def-n-inv}
The \emph{Oikawa invariant of $K_{i}$ and $K_{j}$ in $L$} is the nonnegative integer $n(C_{i},C_{j};\gamma_{0})$ for any nonself-chord $\gamma_{0}$ in $G$ connecting $C_{i}$ and $C_{j}$ $(1\leq i< j\leq3)$. 
It is denoted by $n(K_{i},K_{j})$. 
\end{definition}

\begin{remark}
Consider a Gauss diagram $G$ with three circles $C_{1},C_{2}$, and $C_{3}$ associated with a 3-component classical link diagram (i.e. a link diagram without virtual crossings). 
Fix a nonself-chord $\gamma_{0}$ in $G$ connecting $C_{i}$ and $C_{j}$ $(1\leq i< j\leq3)$. 
It can be seen that all nonself-chords connecting $C_{i}$ and $C_{j}$ are equivalent to $\gamma_{0}$. 
Therefore, if $L=K_{1}\cup K_{2}\cup K_{3}$ is a 3-component classical link (i.e. a link having a classical link diagram), then $n(K_{i},K_{j})=|\mathrm{Lk}(K_{i},K_{j})+\mathrm{Lk}(K_{j},K_{i})|=2|\mathrm{lk}(K_{i},K_{j})|$, 
where $\mathrm{lk}(K_{i},K_{j})$ denotes the classical linking number of $K_{i}$ and $K_{j}$. 
\end{remark}

Now we introduce two invariants $\tau_{0}(L)$ and $\tau_{1}(L)$ of $L$. 
Let $\gamma,\gamma'$, and $\gamma''$ be nonself-chords in $G$ connecting $C_{1}$ and $C_{2}$, $C_{1}$ and $C_{3}$, and $C_{2}$ and $C_{3}$, respectively. 
The endpoints of $\gamma$ and $\gamma'$ on $C_{1}$ divide the circle $C_{1}$ into two arcs. 
Let $\alpha_{1}$ be one of the two arcs. 
Similarly, the endpoints of $\gamma$ and $\gamma''$ on $C_{2}$ divide $C_{2}$ into two arcs, and let $\alpha_{2}$ be one of the two arcs. 
The endpoints of $\gamma'$ and $\gamma''$ on $C_{3}$ divide $C_{3}$ into two arcs, and let $\alpha_{3}$ be one of the two arcs. 
See Figure~\ref{arcs-triple}. 
A triple $(\gamma,\gamma',\gamma'')$ of the three nonself-chords is called \emph{even} if the number of all endpoints of self-/nonself-chords on $\alpha_{1}\cup\alpha_{2}\cup\alpha_{3}$ is an even integer; otherwise it is called \emph{odd}. 
Since $C_{1}, C_{2}$, and $C_{3}$ are even, the parity of $(\gamma,\gamma',\gamma'')$ does not depend on the choice of arcs $\alpha_{1},\alpha_{2}$, and $\alpha_{3}$. 

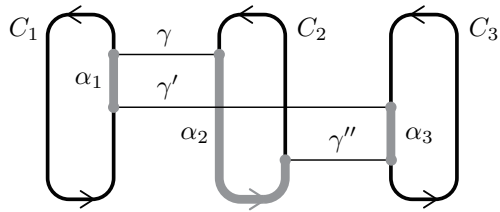
\begin{figure}[htb]
\centering
  \begin{overpic}[width=5.5cm]{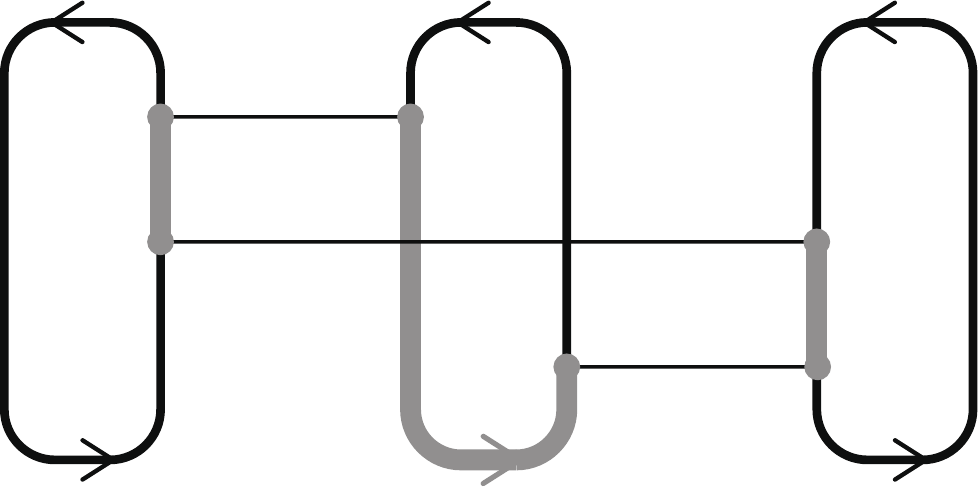}
    \put(-14,66){$C_{1}$}
    \put(95,66){$C_{2}$}
    \put(160,66){$C_{3}$}
    \put(42,64){$\gamma$}
    \put(42,44){$\gamma'$}
    \put(107,24){$\gamma''$}
    \put(11,48){$\alpha_{1}$}
    \put(51,27){$\alpha_{2}$}
    \put(136,27){$\alpha_{3}$}
  \end{overpic}
\caption{A triple of arcs $\alpha_{1}, \alpha_{2}$, and $\alpha_{3}$ for nonself-chords $\gamma,\gamma'$, and $\gamma''$}
\label{arcs-triple}
\end{figure}

Let $T_{0}(G)$ be the set of even triples $(\gamma,\gamma',\gamma'')$ in $G$, and $T_{1}(G)$ the set of odd triples $(\gamma,\gamma',\gamma'')$ in $G$. 
For $k=0,1$, denote by $\tau_{k}(G)$ the sum of signs of all triples in $T_{k}(G)$, where the \emph{sign} of a triple $(\gamma,\gamma',\gamma'')\in T_{k}(G)$ is the product of signs of $\gamma,\gamma'$, and $\gamma''$. 

\begin{example}
Consider the Gauss diagram $G$ given in Example~\ref{ex-n-inv}. 
By definition, we have 
\[
T_{0}(G)=\left\{
\begin{array}{lll}
(\gamma_{0},\gamma_{6},\gamma_{8}), 
(\gamma_{1},\gamma_{6},\gamma_{8}), 
(\gamma_{2},\gamma_{5},\gamma_{8}), 
(\gamma_{2},\gamma_{7},\gamma_{8}), \\
(\gamma_{3},\gamma_{5},\gamma_{8}), 
(\gamma_{3},\gamma_{7},\gamma_{8}), 
(\gamma_{4},\gamma_{5},\gamma_{8}), 
(\gamma_{4},\gamma_{7},\gamma_{8}) 
\end{array}
\right\}
\]
and 
\[
T_{1}(G)=\left\{
\begin{array}{lll}
(\gamma_{0},\gamma_{5},\gamma_{8}), 
(\gamma_{0},\gamma_{7},\gamma_{8}), 
(\gamma_{1},\gamma_{5},\gamma_{8}), 
(\gamma_{1},\gamma_{7},\gamma_{8}), \\
(\gamma_{2},\gamma_{6},\gamma_{8}), 
(\gamma_{3},\gamma_{6},\gamma_{8}), 
(\gamma_{4},\gamma_{6},\gamma_{8}) 
\end{array}
\right\}.
\]
Therefore, it follows that 
$\tau_{0}(G)=4$ and $\tau_{1}(G)=-5$. 
\end{example}

\begin{lemma}\label{lem-tau-inv}
The integers $\tau_{0}(G)$ and $\tau_{1}(G)$ are invariants of $L$. 
Moreover, the difference $\tau_{0}(G)-\tau_{1}(G)$ is invariant under CF-moves. 
\end{lemma}

\begin{proof}
Let $G'$ be a Gauss diagram obtained from $G$ by one of Reidemeister moves and CF-moves. 

For an R1-move or R3-move, we have $T_{k}(G)=T_{k}(G')$ $(k=0,1)$. 
This implies that $\tau_{k}(G)=\tau_{k}(G')$. 

For an R2-move, which adds a pair of chords $\gamma$ and $\gamma'$ with opposite signs, if $\gamma$ and $\gamma'$ are self-chords, then $T_{k}(G)=T_{k}(G')$, and therefore $\tau_{k}(G)=\tau_{k}(G')$ $(k=0,1)$. 
If $\gamma$ and $\gamma'$ are nonself-chords, 
then $T_{k}(G')=T_{k}(G)\cup T_{k}(G';\gamma,\gamma')$ $(k=0,1)$, where $T_{0}(G';\gamma,\gamma')$ denotes the set of even triples in $G'$ that contain $\gamma$ or $\gamma'$, and where $T_{1}(G';\gamma,\gamma')$ denotes the set of odd triples in $G'$ that contain $\gamma$ or $\gamma'$. 
Since the sum of signs of all triples in $T_{k}(G';\gamma,\gamma')$ is zero, we have $\tau_{k}(G)=\tau_{k}(G')$ $(k=0,1)$. 

For a CF-move, which exchanges two consecutive initial/terminal endpoints of a pair of chords $\gamma$ and $\gamma'$, if $\gamma$ and $\gamma'$ are self-chords, then $T_{k}(G)=T_{k}(G')$, and therefore $\tau_{k}(G)=\tau_{k}(G')$ $(k=0,1)$. 
If at least one of $\gamma$ and $\gamma'$ is a nonself-chord, then 
for a triple $t\in T_{k}(G)$ that contains $\gamma$ or $\gamma'$, it follows that $t\in T_{\widetilde{k}}(G')$ $(\widetilde{0}=1, \widetilde{1}=0)$ and the sign of $t$ is altered. 
For a triple $t\in T_{k}(G)$ that does not contain $\gamma$ and $\gamma'$, it follows that $t\in T_{k}(G')$ and the sign of $t$ is preserved $(k=0,1)$. 
Hence, we have $\tau_{0}(G)-\tau_{1}(G)=\tau_{0}(G')-\tau_{1}(G')$. 
\end{proof}

Lemma~\ref{lem-tau-inv} guarantees the well-definedness of the following definition. 

\begin{definition}\label{def-tau-inv}
The \emph{$\tau_{k}$-invariant} of $L$ is the integer $\tau_{k}(G)$ for any Gauss diagram $G$ of $L$ $(k=0,1)$. 
It is denoted by $\tau_{k}(L)$.
\end{definition}

\begin{remark}
Consider a Gauss diagram $G$ with three circles $C_{1},C_{2}$, and $C_{3}$ associated with a 3-component classical link diagram. 
It can be seen that for any three nonself-chords $\gamma, \gamma'$, and $\gamma''$ in $G$ connecting $C_{1}$ and $C_{2}$, $C_{1}$ and $C_{3}$, and $C_{2}$ and $C_{3}$, respectively, their triple $(\gamma,\gamma',\gamma'')$ is even. 
Therefore, if $L$ is classical, then $\tau_{0}(L)=(\mathrm{Lk}(K_{1},K_{2})+\mathrm{Lk}(K_{2},K_{1}))(\mathrm{Lk}(K_{1},K_{3})+\mathrm{Lk}(K_{3},K_{1}))(\mathrm{Lk}(K_{2},K_{3})+\mathrm{Lk}(K_{3},K_{2}))=8\,\mathrm{lk}(K_{1},K_{2})\mathrm{lk}(K_{1},K_{3})\mathrm{lk}(K_{2},K_{3})$ and $\tau_{1}(L)=0$. 
\end{remark}

For convenience of explanation, let $G(a_{12},a_{21};a_{13},a_{31};a_{23},a_{32})$ denote the Gauss diagram $\displaystyle\#_{1\leq i<j\leq3}(G_{ij}(a_{ij})\# G_{ji}(a_{ji}))$ in the case $\mu=3$ given in Proposition~\ref{prop-NF}. 

\begin{lemma}\label{lem-iff-even}
Let $L=K_{1}\cup K_{2}\cup K_{3}$ be the $3$-component virtual link represented by $G(a_{12},a_{21};a_{13},a_{31};a_{23},a_{32})$ for some $a_{ij}\in\mathbb{Z}$. 
Then $L$ is even if and only if 
\[
a_{12}+a_{21}\equiv a_{13}+a_{31}\equiv a_{23}+a_{32}\pmod{2}. 
\]
\end{lemma}

\begin{proof}
Since the number of endpoints on $C_{1}$ is equal to $|a_{12}|+|a_{21}|+|a_{13}|+|a_{31}|$, $C_{1}$ is even if and only if $a_{12}+a_{21}\equiv a_{13}+a_{31}\pmod{2}$. 
Similarly, $C_{2}$ is even if and only if $a_{12}+a_{21}\equiv a_{23}+a_{32}\pmod{2}$. 
Moreover, since the number of endpoints on all circles $C_{1}, C_{2}$, and $C_{3}$ is an even integer, if $C_{1}$ and $C_{2}$ are even, then $C_{3}$ is also even. 
Therefore, we have the conclusion. 
\end{proof}

\begin{lemma}\label{lem-value}
Let $L=K_{1}\cup K_{2}\cup K_{3}$ be the $3$-component virtual link represented by $G(a_{12},a_{21};a_{13},a_{31};a_{23},a_{32})$ for some $a_{ij}\in\mathbb{Z}$. 
If $L$ is even, then we have
$\mathrm{Lk}(K_{i},K_{j})-\mathrm{Lk}(K_{j},K_{i})=a_{ij}-a_{ji}$ and $n(K_{i},K_{j})=|a_{ij}+a_{ji}|$ for any $1\leq i<j\leq3$, and 
$\tau_{0}(L)-\tau_{1}(L)=(a_{12}+a_{21})(a_{13}+a_{31})(a_{23}+a_{32})$. 
\end{lemma}

\begin{proof}
Since the sum of signs of all nonself-chords of type~$(i,j)$ is equal to $a_{ij}$ and that of type~$(j,i)$ is equal to $a_{ji}$, we have $\mathrm{Lk}(K_{i},K_{j})-\mathrm{Lk}(K_{j},K_{i})=a_{ij}-a_{ji}$ $(1\leq i<j\leq3)$. 

Choose and fix a nonself-chord $\gamma_{0}$ connecting $C_{i}$ and $C_{j}$ $(1\leq i<j\leq3)$. 
By definition, all nonself-chords connecting $C_{i}$ and $C_{j}$ are equivalent to $\gamma_{0}$. 
The sum of their signs is equal to $a_{ij}+a_{ji}$. 
Therefore, since $\sigma(C_{i},C_{j};\gamma_{0})=a_{ij}+a_{ji}$ and $\overline{\sigma}(C_{i},C_{j};\gamma_{0})=0$, 
it follows that $n(K_{i},K_{j})=|a_{ij}+a_{ji}|$. 

For any three nonself-chords $\gamma, \gamma'$, and $\gamma''$ in $G=G(a_{12},a_{21};a_{13},a_{31};a_{23},a_{32})$ connecting $C_{1}$ and $C_{2}$, $C_{1}$ and $C_{3}$, and $C_{2}$ and $C_{3}$, respectively, their triple $(\gamma,\gamma',\gamma'')$ is even by definition. 
Therefore, we have $\tau_{0}(G)=(a_{12}+a_{21})(a_{13}+a_{31})(a_{23}+a_{32})$ and $\tau_{1}(G)=0$, which implies that $\tau_{0}(L)-\tau_{1}(L)=(a_{12}+a_{21})(a_{13}+a_{31})(a_{23}+a_{32})$. 
\end{proof}

\begin{lemma}\label{lem-equiv-Gauss}
We have the following CF-equivalent Gauss diagrams of $3$-component virtual links. 
\begin{enumerate}
\item $G(a_{12},a_{21};a_{13},a_{31};a_{23},a_{32})\sim G(-a_{21},-a_{12};-a_{31},-a_{13};a_{23},a_{32})$. 
\item $G(a_{12},a_{21};a_{13},a_{31};a_{23},a_{32})\sim G(-a_{21},-a_{12};a_{13},a_{31};-a_{32},-a_{23})$. 
\item $G(a_{12},a_{21};a_{13},a_{31};a_{23},a_{32}) \sim G(a_{12},a_{21};-a_{31},-a_{13};-a_{32},-a_{23})$. 
\end{enumerate}
\end{lemma}

\begin{proof}
Since the proofs of (i), (ii), and (iii) are similar, we only show (i). 
Figure~\ref{pf-lem-equiv-Gauss} indicates the proof. 
More precisely, (2) is obtained from (1) by an R1-move adding a free chord $\gamma$ in $C_{1}$. 
Applying Lemma~\ref{lem-HT}, we slide the terminal endpoint of $\gamma$ along $C_{1}$ with respect to the orientation of $C_{1}$ until it is next to the initial endpoint. 
Then (3) is obtained from (2). 
We delete $\gamma$ by an R1-move, and use Lemma~\ref{lem-HT} to exchange the positions of the nonself-chords of type~$(i,j)$ and those of type~$(j,i)$ for $(i,j)=(1,2), (1,3)$. 
Finally (4) is obtained from (3). 
\end{proof}

\begin{figure}[htb]
\centering
\vspace{1em}
  \begin{overpic}[width=12cm]{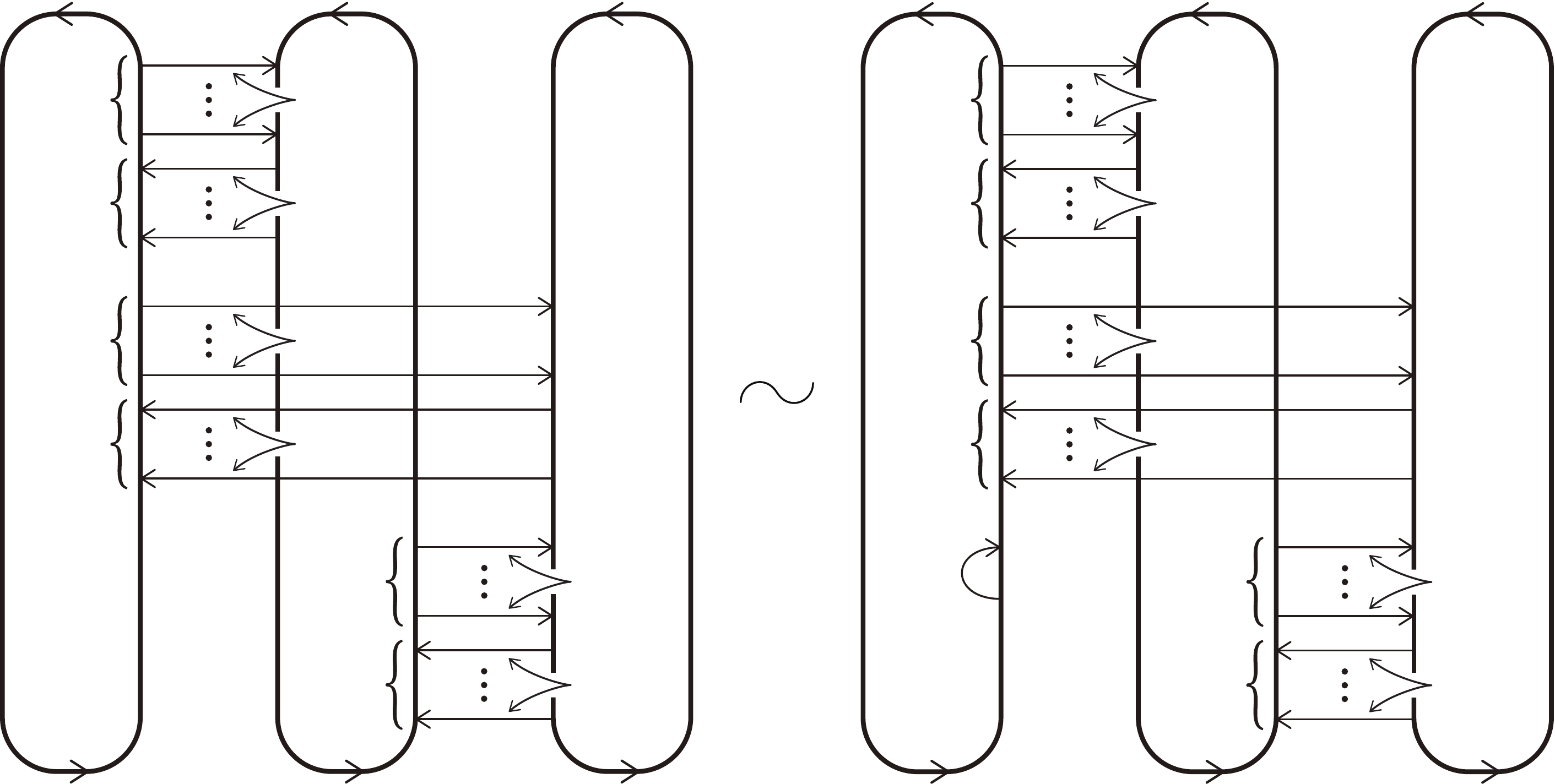}
    \put(8,-14){(1) $G(a_{12},a_{21};a_{13},a_{31};a_{23},a_{32})$}
    \put(258.5,-14){(2)}
    \put(164.5,92){R1}
    \put(10,176){$C_{1}$}
    \put(71,176){$C_{2}$}
    \put(131,176){$C_{3}$}
    \put(4,148){$|a_{12}|$}
    \put(4,125){$|a_{21}|$}
    \put(4,95){$|a_{13}|$}
    \put(4,72){$|a_{31}|$}
    \put(68,148){$\e_{12}$}
    \put(68,125){$\e_{21}$}
    \put(68,95){$\e_{13}$}
    \put(68,72){$\e_{31}$}
    \put(64.5,42){$|a_{23}|$}
    \put(64.5,19){$|a_{32}|$}
    \put(128.5,42){$\e_{23}$}
    \put(128.5,19){$\e_{32}$}
    
    
    \put(199,176){$C_{1}$}
    \put(259,176){$C_{2}$}
    \put(319,176){$C_{3}$}
    \put(202,45){$\gamma$}
    \put(193,148){$|a_{12}|$}
    \put(193,125){$|a_{21}|$}
    \put(193,95){$|a_{13}|$}
    \put(193,72){$|a_{31}|$}
    \put(257,148){$\e_{12}$}
    \put(257,125){$\e_{21}$}
    \put(257,95){$\e_{13}$}
    \put(257,72){$\e_{31}$}
    \put(253.5,42){$|a_{23}|$}
    \put(253.5,19){$|a_{32}|$}
    \put(317.5,42){$\e_{23}$}
    \put(317.5,19){$\e_{32}$}
  \end{overpic}

\vspace{2.5em}
\centering
  \begin{overpic}[width=12cm]{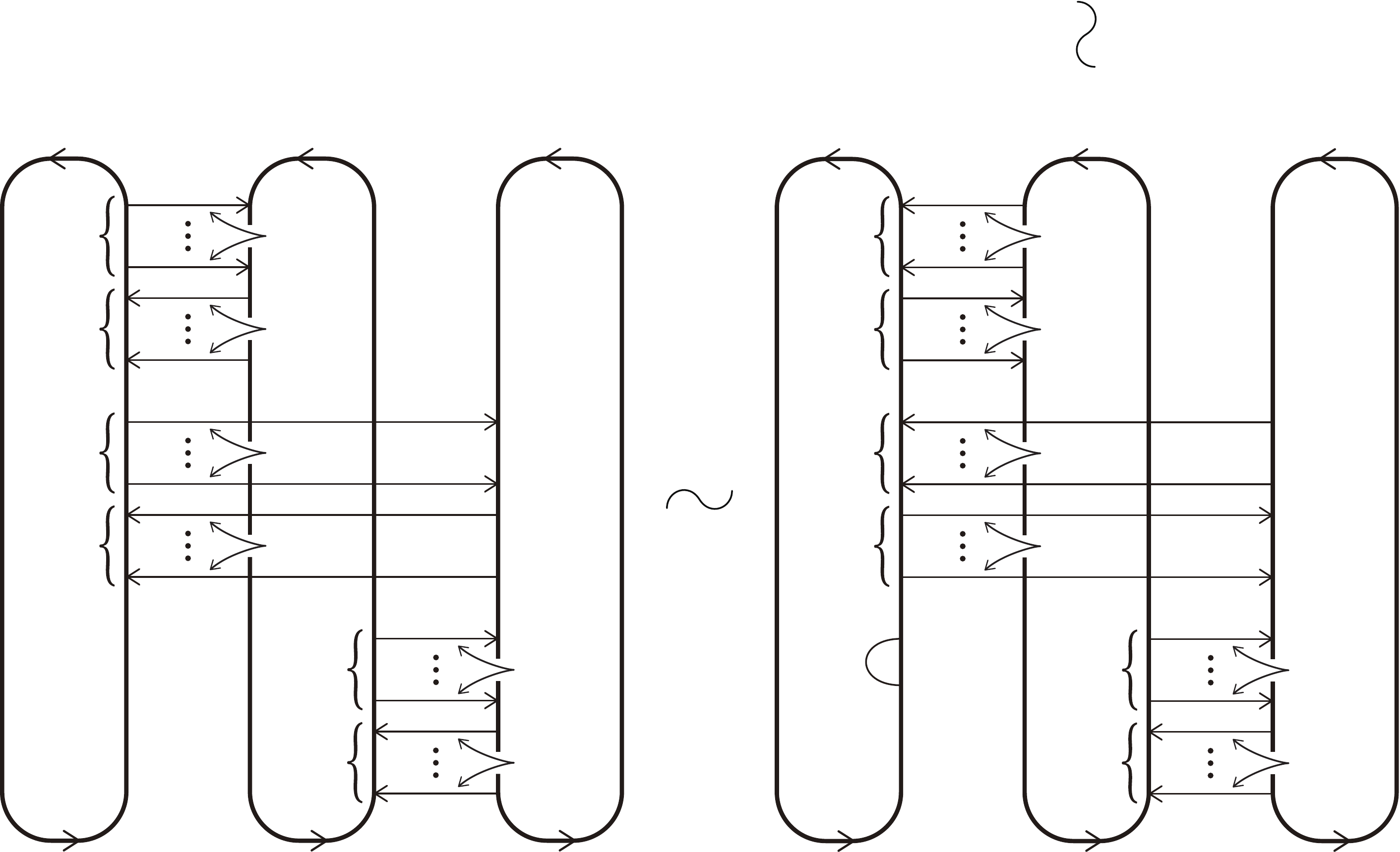}
    \put(-8,-14){(4) $G(-a_{21},-a_{12};-a_{31},-a_{13};a_{23},a_{32})$}
    \put(258.5,-14){(3)}
    \put(164.5,92){R1}
    \put(272,197){Lem \ref{lem-HT}}
    \put(153.25,73){Lem \ref{lem-HT}}
    \put(10,176){$C_{1}$}
    \put(71,176){$C_{2}$}
    \put(131,176){$C_{3}$}
    \put(4,148){$|a_{21}|$}
    \put(4,125){$|a_{12}|$}
    \put(4,95){$|a_{31}|$}
    \put(4,72){$|a_{13}|$}
    \put(68,148){$-\e_{21}$}
    \put(68,125){$-\e_{12}$}
    \put(68,95){$-\e_{31}$}
    \put(68,72){$-\e_{13}$}
    \put(64.5,42){$|a_{23}|$}
    \put(64.5,19){$|a_{32}|$}
    \put(128.5,42){$\e_{23}$}
    \put(128.5,19){$\e_{32}$}
    
    
    \put(199,176){$C_{1}$}
    \put(259,176){$C_{2}$}
    \put(319,176){$C_{3}$}
    \put(202,45){$\gamma$}
    \put(193,148){$|a_{12}|$}
    \put(193,125){$|a_{21}|$}
    \put(193,95){$|a_{13}|$}
    \put(193,72){$|a_{31}|$}
    \put(257,148){$-\e_{12}$}
    \put(257,125){$-\e_{21}$}
    \put(257,95){$-\e_{13}$}
    \put(257,72){$-\e_{31}$}
    \put(253.5,42){$|a_{23}|$}
    \put(253.5,19){$|a_{32}|$}
    \put(317.5,42){$\e_{23}$}
    \put(317.5,19){$\e_{32}$}
  \end{overpic}
\vspace{1em}
\caption{Proof of Lemma~\ref{lem-equiv-Gauss}(i)}
\label{pf-lem-equiv-Gauss}
\end{figure}
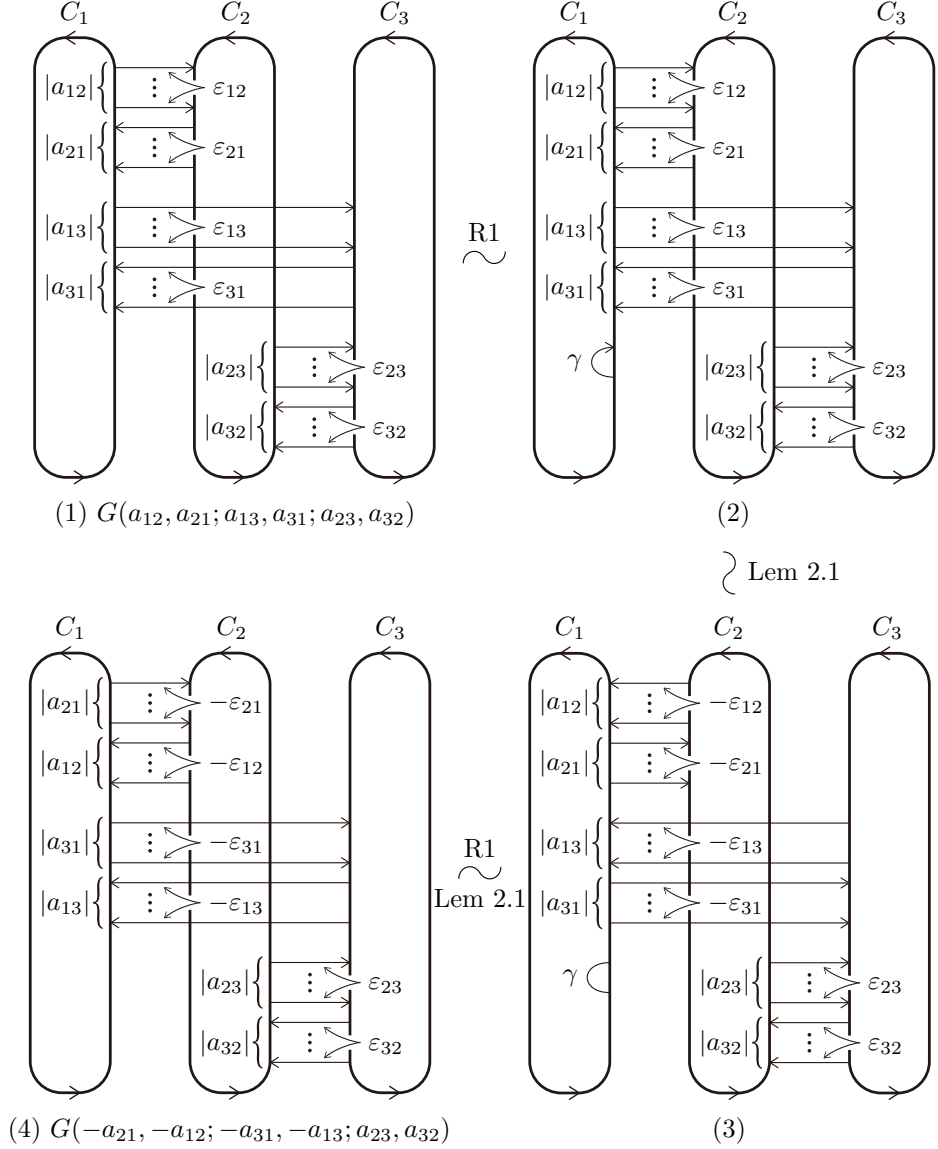

\begin{proof}[Proof of Theorem~\ref{th-even}]
$\mathrm{(i)}\Rightarrow\mathrm{(ii)}$: 
This follows from Lemmas~\ref{lem-Lk}, \ref{lem-n-inv}, and \ref{lem-tau-inv} directly. 

\noindent
$\mathrm{(ii)}\Rightarrow\mathrm{(i)}$: 
Let $G$ and $G'$ be Gauss diagrams of $L$ and $L'$, respectively. 
By Proposition~\ref{prop-NF}, we have 
\[
G\sim G(a_{12},a_{21};a_{13},a_{31};a_{23},a_{32})
\]
for some $a_{ij}\in\mathbb{Z}$ and 
\[
G'\sim G(a'_{12},a'_{21};a'_{13},a'_{31};a'_{23},a'_{32})
\]
for some $a'_{ij}\in\mathbb{Z}$. 
Since $\mathrm{Lk}(K_{i},K_{j})-\mathrm{Lk}(K_{j},K_{i})=\mathrm{Lk}(K'_{i},K'_{j})-\mathrm{Lk}(K'_{j},K'_{i})$ and $n(K_{i},K_{j})=n(K'_{i},K'_{j})$, by Lemma~\ref{lem-value} we have $a_{ij}-a_{ji}=a'_{ij}-a'_{ji}$ and $|a_{ij}+a_{ji}|=|a'_{ij}+a'_{ji}|$
$(1\leq i<j\leq3)$. 
Therefore, there are the following cases to consider. 
\begin{enumerate}
\renewcommand{\labelenumi}{(\arabic{enumi})}
\item $a_{12}=a'_{12},a_{21}=a'_{21},a_{13}=a'_{13},a_{31}=a'_{31},a_{23}=a'_{23},a_{32}=a'_{32}$. \label{case1} 
\item $a_{12}=-a'_{21},a_{21}=-a'_{12},a_{13}=-a'_{31},a_{31}=-a'_{13},a_{23}=a'_{23},a_{32}=a'_{32}$. \label{case2} 
\item $a_{12}=-a'_{21},a_{21}=-a'_{12},a_{13}=a'_{13},a_{31}=a'_{31},a_{23}=-a'_{32},a_{32}=-a'_{23}$. \label{case3} 
\item $a_{12}=a'_{12},a_{21}=a'_{21},a_{13}=-a'_{31},a_{31}=-a'_{13},a_{23}=-a'_{32},a_{32}=-a'_{23}$. \label{case4} 
\item $a_{12}=-a'_{21},a_{21}=-a'_{12},a_{13}=a'_{13},a_{31}=a'_{31},a_{23}=a'_{23},a_{32}=a'_{32}$. \label{case5} 
\item $a_{12}=a'_{12},a_{21}=a'_{21},a_{13}=-a'_{31},a_{31}=-a'_{13},a_{23}=a'_{23},a_{32}=a'_{32}$. \label{case6} 
\item $a_{12}=a'_{12},a_{21}=a'_{21},a_{13}=a'_{13},a_{31}=a'_{31},a_{23}=-a'_{32},a_{32}=-a'_{23}$. \label{case7} 
\item $a_{12}=-a'_{21},a_{21}=-a'_{12},a_{13}=-a'_{31},a_{31}=-a'_{13},a_{23}=-a'_{32},a_{32}=-a'_{23}$. \label{case8} 
\end{enumerate}
Since $\tau_{0}(L)-\tau_{1}(L)=\tau_{0}(L')-\tau_{1}(L')$, Lemma~\ref{lem-value} implies that 
\[
(a_{12}+a_{21})(a_{13}+a_{31})(a_{23}+a_{32})
=(a'_{12}+a'_{21})(a'_{13}+a'_{31})(a'_{23}+a'_{32}).
\]
This indicates that Cases~(\ref{case5})--(\ref{case8}) do not occur. 
Hence, it is enough to consider Cases~(\ref{case1})--(\ref{case4}). 

In Case~(\ref{case1}), we have 
\[
G\sim G(a_{12},a_{21};a_{13},a_{31};a_{23},a_{32})
=G(a'_{12},a'_{21};a'_{13},a'_{31};a'_{23},a'_{32})
\sim G'. 
\] 
In Case~(\ref{case2}), it follows from Lemma~\ref{lem-equiv-Gauss}(i) that 
\begin{eqnarray*}
G&\sim& G(a_{12},a_{21};a_{13},a_{31};a_{23},a_{32}) \\
&=&G(-a'_{21},-a'_{12};-a'_{31},-a'_{13};a'_{23},a'_{32}) \\
&\sim& G(a'_{12},a'_{21};a'_{13},a'_{31};a'_{23},a'_{32})
\sim G'. 
\end{eqnarray*}
Cases~(\ref{case3}) and~(\ref{case4}) are shown similarly. 
\end{proof}

We conclude this paper with a result that gives a relation among the classifying invariants $\mathrm{Lk}(K_{i},K_{j})-\mathrm{Lk}(K_{j},K_{i})$, $n(K_{i},K_{j})$, and $\tau_{0}(L)-\tau_{1}(L)$ of Theorem~\ref{th-even}. 

\begin{proposition}\label{prop-relation}
Let $L=K_{1}\cup K_{2}\cup K_{3}$ be a $3$-component even virtual link, and let $x_{ij},y_{ij}$, and $z$ be integers $(1\leq i<j\leq3)$. 
 Assume that $\mathrm{Lk}(K_{i},K_{j})-\mathrm{Lk}(K_{j},K_{i})=x_{ij}$, $n(K_{i},K_{j})=|y_{ij}|$, and $\tau_{0}(L)-\tau_{1}(L)=z$. 
Then we have 
\[
x_{12}\equiv x_{13}\equiv x_{23}\equiv y_{12}\equiv y_{13}\equiv y_{23}\pmod{2}\] 
and 
\[|z|=|y_{12}y_{13}y_{23}|.\] 
\end{proposition}

\begin{proof}
By Proposition~\ref{prop-NF}, $L$ is CF-equivalent to a 3-component even virtual link represented by $G(a_{12},a_{21};a_{13},a_{31};a_{23},a_{32})$ for some $a_{ij}\in\mathbb{Z}$. 
By Lemmas~\ref{lem-Lk}, \ref{lem-n-inv}, \ref{lem-tau-inv}, and \ref{lem-value}, we have $\mathrm{Lk}(K_{i},K_{j})-\mathrm{Lk}(K_{j},K_{i})=a_{ij}-a_{ji}$, $n(K_{i},K_{j})=|a_{ij}+a_{ji}|$ $(1\leq i<j\leq3)$, and $\tau_{0}(L)-\tau_{1}(L)=(a_{12}+a_{21})(a_{13}+a_{31})(a_{23}+a_{32})$. 
The assumption implies that 
\[
x_{ij}=a_{ij}-a_{ji},\, |y_{ij}|=|a_{ij}+a_{ji}|\, (1\leq i<j\leq3), 
\]
and 
\[
z=(a_{12}+a_{21})(a_{13}+a_{31})(a_{23}+a_{32}).
\] 
Since $L$ is even, we have $a_{12}+a_{21}\equiv a_{13}+a_{31}\equiv a_{23}+a_{32}\pmod{2}$ by Lemma~\ref{lem-iff-even}. 
Therefore, it follows that 
\[
x_{12}\equiv x_{13}\equiv x_{23}\equiv y_{12}\equiv y_{13}\equiv y_{23}\pmod{2}. 
\]
Moreover, since $|y_{ij}|=|a_{ij}+a_{ji}|$ $(1\leq i<j\leq3)$ and $z=(a_{12}+a_{21})(a_{13}+a_{31})(a_{23}+a_{32})$, we have $|z|=|y_{12}y_{13}y_{23}|$. 
\end{proof}



\end{document}